\documentclass[12pt]{amsproc}

\usepackage{fullpage}
\usepackage{amsmath}
\usepackage{amsfonts}
\usepackage{amssymb}
\usepackage{amsthm}
\usepackage[utf8]{inputenc}
\usepackage{breqn}
\usepackage{afterpage}
\usepackage{longtable}
\usepackage{indentfirst}
\usepackage{caption}
\usepackage{subcaption}
\usepackage{graphicx}
\graphicspath{ {./images/} }
\newtheorem{theorem}{Theorem}[section]
\newtheorem{lemma}[theorem]{Lemma}

\theoremstyle{definition}

\newtheorem{theorem-definition}[theorem]{Theorem-Definition}

\theoremstyle{remark}
\newtheorem{remark}[theorem]{Remark}

\numberwithin{equation}{section}

\usepackage{color}
\usepackage{xcolor}

\title{On the capacity dimensions of the Brjuno and Perez-Marco sets}\begin{author}[N.~Akramov]{Nurali Akramov}
    \address{National University of Uzbekistan,  Tashkent, Uzbekistan}
\email{nurali.akramov.1996$@$gmail.com}
\end{author}
\begin{author}[A.~Ashirov]{Abduvahhob Ashirov}

\address{National University of Uzbekistan,  Tashkent, Uzbekistan}
\email{ashurovabduvahob007$@$gmail.com}
\end{author}

\date{}

\begin{document}

\begin{abstract}
    In this work, we prove that the complement of the Brjuno set $\mathcal{B}$ has a zero  capacity with respect to the kernel $k^1_\sigma(z,\xi)=\ln^2{|z-\xi|}\left|\ln{\ln{\left(e+\frac{1}{|z-\xi|}\right)}}\right|^\sigma$ for any
 $\sigma > 2$. Similarly, the complement of the Perez-Marco set $\mathcal{PM}$ has a zero capacity with respect to the kernel $
k^2_\sigma(z,\xi) = \ln^{2}{\ln\left(e+\frac{1} {\left| {z - \xi }\right|}\right)}
\cdot\ln^{\sigma}{\ln\ln\left(e^3+\frac{1} {\left| {z - \xi }\right|}\right)}$ for any $\sigma>2$.
\end{abstract}
\maketitle 
\section{Introduction}
The linearization of holomorphic maps near the fixed point is one of the main topics in local dynamical systems. Consider a holomorphic germ 
$$f(z)=\lambda z+c_2z^2+c_3z^3+...+c_dz^d+\dots.$$ 
We say that $f$ is \emph{linearizable} in a neighborhood of fixed point $z=0$ if there exists a germ of holomorphic function $\varphi$ (with $\varphi'(0)\not=0$) such that
$$\varphi(z)\circ f(z)\circ\varphi^{-1}(z)=\lambda{z}.$$

 When $|\lambda|\not=0,1$, such a linearization always exists (see \cite{GK}). For $\lambda=e^{2\pi{i}\alpha}$ with an irrational $\alpha\in\mathbb{R}$, A. Brjuno established the sharp linearizability condition.
 \begin{theorem}[{Brjuno \cite{ABR}, Yoccoz \cite{YOC}}]
If $\alpha$ is a Brjuno number,  $f$ is linearizable. If $\alpha$ is a Brjuno number, then $f(z)=e^{2\pi i\alpha}z+z^2$ is not linarizable an it has the property that every neighborhood of the origin contains infinitely many periodic orbits.    
 \end{theorem}
Perez-Marco completely characterized the multipliers for which such periodic orbits must appear.

\begin{theorem}[{Perez-Marco \cite{MJ}}] If $\alpha$ is a Perez-Marco number, $\alpha\in \mathcal{PM}$, then
any non linearizable germ with multiplier $e^{2\pi \alpha i}$
contains infinitely many periodic orbits in every neighborhood of the
origin. Moreover, whenever $\alpha$ is not a Perez-Marco number there exists a non-linearizable germ with multiplier $e^{2\pi \alpha i}$ which has no  periodic orbit other than the fixed point itself within some neighborhood of the fixed point.
\end{theorem}

To define the Brjuno set and the Perez-Marco set, we need the continued fraction expansion of $\alpha$ obtained via the Gauss map $A(\alpha):(0,1]\to[0,1]$ defined as 
$$A(\alpha)=\frac{1}{\alpha}-\left[\frac{1}{\alpha}\right]$$
where $[x]$ the whole part of the number $x$. For $\alpha\in\mathbb{R}\setminus\mathbb{Q}$, we define $\alpha_0=\alpha-[\alpha]$ and then we continue by induction: $\alpha_{n+1}=A(\alpha_n)$, $a_{n+1}=[\frac{1}{\alpha_n}]$. 

Since $\alpha$ is an irrational number, this process produces infinite continued fraction expansion and so, $\alpha$ can be written as
$$\alpha=a_0+\cfrac{1}{a_1+\cfrac{1}{a_2+\dots\cfrac{1}{a_n+\dots}}}=:[a_0,a_1,...,a_n,...]$$
where $a_j$ are positive integers. 

The following fraction is called $n$th convergent of $\alpha$ an this fraction is a rational 
\begin{equation}\label{eq:Qn}
 \frac{P_n}{Q_n}:=[a_0,a_1,...,a_n].   
\end{equation}

 An irrational number $\alpha\in\mathbb{R}$ is called \emph{the Brjuno number} if its sequence of continued fraction denominators $Q_n$ satisfies
\begin{equation}
\sum_{n=1}^\infty\frac{\ln{Q_{n+1}}}{Q_n}<+\infty,
\end{equation}
and we say that $\alpha$ is \emph{Perez-Marco number} if 
\begin{equation}
\sum_{n=1}^\infty\frac{\ln{\ln{Q_{n+1}}}}{Q_n}<+\infty.
\end{equation}
The set of all Brjuno numbers, denoted by $\mathcal{B}$, is called \emph{the Brjuno set} while the set of all Perez-Marco numbers, denoted by $\mathcal{PM}$, is called \emph{the Perez-Marco set}.

Our first main result is capacity dimension of the Brjuno set (see Section 2 for definitions of  $C_\sigma$-capacity and the Hausdorff $H^h$-measure).

\begin{theorem}\label{t:mainB}
The complement $ {\mathbb{R}}
\backslash \mathcal{B}$  of the Brjuno set has zero $
C_\sigma $ -capacity with respect to the kernel
$$k^1_\sigma(z,\xi)=\ln^2{|z-\xi|}\left|\ln{\ln{\left(e+\frac{1}{|z-\xi|}\right)}}\right|^\sigma,\,\,
 \sigma > 2.$$
In particular, it
has zero Hausdorff $ H^{h^1_\sigma}$-measure with respect to the
gauge function $$ h_\sigma^1\left( t \right) =  \ln^{-2} \frac{1}{t}\cdot\ln^{ -
\sigma }{ \ln \frac{1}{t}},\,\,\,\sigma
> 2,\,\,\,0\leq t\leq e^{-2}.$$
\end{theorem}

Similarly, we have the following result for the Perez-Marco set.
\begin{theorem}\label{t:mainPM}

The complement $ {\mathbb{R}}
\backslash \mathcal{PM}$  of the Perez-Marco set has zero $
C_\sigma $-capacity with respect to the kernel
$$
k^2_\sigma(z,\xi) = \ln^{2}{\ln\left(e+\frac{1} {\left| {z - \xi }\right|}\right)}
\cdot\ln^{\sigma}{\ln\ln\left(e^3+\frac{1} {\left| {z - \xi }\right|}\right)},\,\,
 \sigma > 2.$$
In particular, it
has zero Hausdorff $ H^{h^2_\sigma}$-measure with respect to the
gauge function   $$h_\sigma^2\left( t \right) = \ln^{ -
2}{\ln \frac{1}{t}}\cdot\ln^{ -
\sigma}\ln{\ln \frac{1}{t}},\,\,\,\sigma
> 2,\,\,\,0\leq t\leq e^{-10}.$$
\end{theorem}

 Theorem \ref{t:mainB} improves the result obtained by A. Sadullaev and K. Rakhimov about the capacity dimension of the Brjuno set (see \cite{AS1}). While,  Theorem \ref{t:mainPM} enhances K. Rakhimov's result on the capacity of  the Perez-Marco set (see \cite{KR}).

\section{Preliminary}

\subsection{Haussdorf $H^h$-measure} 
Consider a bound set $E\subset\mathbb{R}^n$ and strictly increasing continues function $h:[0,r_0]\to[0,+\infty)$, where $r_0>0$ and $h(0)=0$. For a fixed $\varepsilon>0$, take any covering of $E$ by finite family of open balls $B_j(x_j,r_j)$ where $1\le{j}\le{m}$ and $r_j<\varepsilon$. Define

$$H^{h}(E,\varepsilon)=\inf\left\{\sum_{j=1}^m h(r_j):\text{ } \bigcup_{j=1}^m{B_j}\supset E \right\}.$$
Since,  $H^{h}(E,\varepsilon)$ non-decreasing as $\varepsilon$ decreases, the limit
$$H^{h}(E)=\lim_{\varepsilon\to0+}H^{h}(E,\varepsilon)$$
exists. This limit is called the \textit{Hausdorff $H^h$-measure} of $E$.  If $h_{\alpha}(t)=t^\alpha$ for some  $\alpha>0$, then the measure $H^{h_\alpha}(E)$ becomes the standard $\alpha$-dimensional Hausdorff measure $H^\alpha(E)$.

\subsection{$C_\sigma$-capacity}
Capacities are one of the important tools in potential and pluripotential theory.  Many researchers, including A. Sadullaev, S. Imomkulov, E. Bedford, and B.A. Taylor, K. Rakhimov have made significant contributions in this domain and applied to various areas (see \cite{KR2},\cite{ERB},\cite{AS1},\cite{AS2}). In this section, we define a $\mathbb C^n$-capacity as in \cite{NSL2}, such that its null set is slightly larger than polar sets.

Let $K\subset\mathbb{C}$ be a compact set. We define the kernel $$k^1_\sigma(z,\xi)=\ln^2{|z-\xi|}\left|\ln{\ln{\left(e+\frac{1}{|z-\xi|}\right)}}\right|^\sigma$$
for any $\sigma>0$. 

For a positive Borel measure $\mu$ supported on $K$, define its potential in a point $z$ 
$$U_\sigma^1(z)=\int_K k_\sigma^1(z,\xi)d\mu(\xi).$$
We also define 
$$I_\sigma^1(\mu)=\int_KU_\sigma^1(z)d\mu(z)=\iint_{K\times K}k_\sigma^1(z,\xi)d\mu(z)d\mu(\xi) $$
and let $W^1_\sigma(K)=\inf\{I^1_\sigma(\mu)\}$ where $\mu$ is a probability measure supported on $K$.
The $C^1_\sigma$-capacity of $K$ is then defined by $$C^1_\sigma(K)=\left(W^1_\sigma(K)\right)^{-1}.$$
We also define the inner and outer $C^1_\sigma$-capacities with compact subsets and open supersets of any Borel set $E$
$$\underline{C}_\sigma(E)=\sup\{C_\sigma(K): K\subset{E},\text{ }K-compact\},$$
$$\overline{C}_\sigma(E)=\inf\{C_\sigma(U): U\supset{E},\text{ }U-open\}.$$

Analogously, using the kernel

$$k^2_\sigma(z,\xi) = \ln^{2}{\ln\left(e+\frac{1} {\left| {z - \xi }\right|}\right)}
\cdot\ln^{\sigma}{\ln\ln\left(e^3+\frac{1} {\left| {z - \xi }\right|}\right)}$$
for $\sigma>2$, we define $U_\sigma^2$, $I_\sigma^2$, $W_\sigma^2(K)$, and hence $C_\sigma^2$-capacity as
$$C^2_\sigma(K)=(W(K))^{-1},$$
with the same conventions for inner and outer capacities.

We note common properties for either capacities (see \cite{NSL1}, \cite{NSL2}, \cite{LCA}): 
\begin{enumerate}
    \item  For every Borel set $E\subset\mathbb{C}$ inner and outer $C_\sigma$-capacities coincide $$ \overline{C}_\sigma(E)=\underline{C}_\sigma(E)=C_\sigma(E).$$
    \item $C_\sigma$-capacity of a Borel set is equal to zero if and only if there exist a finite Borel measure $\mu$ such that
$$U_\sigma^\mu(z)\equiv+\infty,\text{ }\forall{z}\in{E}.$$    
    \item $C_\sigma$-capacity is subadditive, that is, if $E=\cup_{j=1}^\infty E_j$, then
    $$C_\sigma(E)\le\sum_{j=1}^\infty C_\sigma(E_j).$$
    \item Connection to the Hausdorff $H^h$-measure (see \cite{NSL2}):
  
If $C_\sigma^1(E)=0$ then the Hausdorff $H^h$-measure is equal to zero with respect to the gauge function 
$$ h_\sigma^1\left( t \right) =  \ln^{-2} \frac{1}{t}\cdot\ln^{ -
\sigma }{ \ln \frac{1}{t}},\,\,\,\sigma
> 2,\,\,\,0\leq t\leq e^{-2}.$$

Similarly, $C_\sigma^2(E)=0$ implies Hausdorff $H^h$-measure is zero with respect to the gauge function 
 $$h_\sigma^2\left( t \right) = \ln^{ -
2}{\ln \frac{1}{t}}\cdot\ln^{ -
\sigma}\ln{\ln \frac{1}{t}},\,\,\,\sigma
> 2,\,\,\,0\leq t\leq e^{-10}.$$ 

\end{enumerate}

\subsection{Some properties of continued fractions}  Let $\alpha$ be an irrational number, and $\frac{P_n}{Q_n}$ be as in \eqref{eq:Qn}. Let us recall some  properties of  $Q_n$ which we use in the sequel (see \cite{SMA}):
\begin{enumerate}
    \item\label{z.kasr1} For any $n\geq 1$, we have  
    $$Q_{n}\geq \frac{1}{2}\left(\frac{\sqrt{5}+1}{2}\right)^{n-1}.$$
    \item\label{z.kasr2} For any $n\geq 1$, we have 
    $$\frac{1}{2Q_{n}\cdot Q_{n+1}}<\left|\alpha-\frac{P_n}{Q_{n}}\right|<\frac{1}{Q_{n}\cdot Q_{n+1}}.$$
\end{enumerate}

\section{Proof of Theorem \ref{t:mainB} and Theorem \ref{t:mainPM}}

Let us first prove the following lemma.

\begin{lemma}\label{claim:brjuno} Let $\alpha$ be an irrational number and $Q_n$ be as in \eqref{eq:Qn}. If $$\sum_{n=1}^\infty\frac{\ln{Q_{n+1}}}{Q_n}=+\infty$$
  then, for any $\varepsilon>0$, we have
       $$ \ \sum_{n=2}^\infty\frac{\ln^2 Q_{n+1}(\ln{\ln{Q_{n+1}}})^{2+4\varepsilon}}{Q^2_n \ln^{1+\varepsilon}{Q_{n}}}=+\infty.$$ 
\end{lemma}

   \begin{proof} Consider the set $$\mathcal{N}:=\left\{n\in\mathbb{N} :\ (\ln\ln Q_{n+1})^{2+4\varepsilon}<(\ln Q_{n})^{2+3\varepsilon} \right\}.$$  
  Then we have 
    $$\sum_{n=2}^\infty\frac{\ln^2 Q_{n+1}(\ln{\ln{Q_{n+1}}})^{2+4\varepsilon}}{Q^2_n \ln^{1+\varepsilon}{Q_{n}}}=\sum_{n\in\mathcal{N}}\frac{\ln^2 Q_{n+1}(\ln{\ln{Q_{n+1}}})^{2+4\varepsilon}}{Q^2_n \ln^{1+\varepsilon}{Q_{n}}}+\sum_{n\notin\mathcal{N}}\frac{\ln^2 Q_{n+1}(\ln{\ln{Q_{n+1}}})^{2+4\varepsilon}}{Q^2_n \ln^{1+\varepsilon}{Q_{n}}}$$
    Let $\beta=\frac{2+3\varepsilon}{2+4\varepsilon}<1$. Then by the definition of the set $\mathcal{N}$ we have 
    $$\ln Q_{n+1}<e^{(\ln Q_{n})^\beta}$$ for any $n\in\mathcal{N}$. 
    Also, from the property \ref{z.kasr2} of continued fraction follows that, there exists $n_0'\in \mathbb{N}$ such that  $$(\ln{Q_{n}})^\beta<\frac{1}{2}\ln Q_{n},\quad \forall n\geq n'_0.$$  
    Using these inequalities, we come to the following conclusion: 
    \begin{align*} \sum_{n\in\mathcal{N},n\geq n'_0}\frac{\ln{Q_{n+1}}}{Q_{n}}&<\sum_{n\in\mathcal{N},n\geq n'_0}\frac{e^{(\ln{Q_{n}})^\beta}}{Q_{n}}=\sum_{n\in\mathcal{N},n\geq n'_0}\frac{e^{(\ln{Q_{n}})^\beta-\frac{1}{2}\ln Q_{n}}}{Q^{\frac{1}{2}}_n}\\&<\sum_{n\in\mathcal{N},n\geq n'_0}\frac{1}{Q^{\frac{1}{2}}_n}<+\infty.   
    \end{align*}
    
Consequently,  it follows that $$\sum_{n\notin\mathcal{N}}\frac{\ln{Q_{n+1}}}{Q_{n}}=+\infty.$$ 
    Then, by the Cauchy-Bunyakovsky inequality, we have  
\begin{align*}
+\infty&=\sum_{n\notin\mathcal{N}}\frac{\ln Q_{n+1}}{Q_{n} }\\
&=\sum_{n\notin\mathcal{N}}\frac{\ln Q_{n+1}\ln^{\frac{1}{2}+\varepsilon}{Q_{n}}}{Q_{n}\ln^{\frac{1}{2}+\varepsilon}{Q_{n}} }\\ 
&\leq \left(\sum_{n\notin\mathcal{N}}\frac{\ln^2 Q_{n+1}\ln^{1+2\varepsilon}{Q_{n}}}{Q^2_n }\right)^{\frac{1}{2}}\left(\sum_{n\notin\mathcal{N}}\frac{1}{\ln^{1+2\varepsilon}{Q_{n}} }\right)^{\frac{1}{2}}.
\end{align*}

 Again, from property \ref{z.kasr1} of continued fractions there exists a constant $C>0$ such that for all $ n\geq 2$ we have $\ln{Q_{n}}>Cn$. So, we have 
\begin{align*}
        \sum_{n=2}^\infty\frac{1}{\ln^{1+2\varepsilon}{Q_{n}} }<\frac{1}{C^{1+2\varepsilon}}\sum_{n=2}^\infty\frac{1}{n^{1+2\varepsilon} }<+\infty.
\end{align*} 
This implies that  \begin{align*}\sum_{n\notin\mathcal{N},n\geq 2}\frac{\ln^2 Q_{n+1}\ln^{1+2\varepsilon}{Q_{n}}}{Q^2_n }=+\infty.
\end{align*} 
Finally, we have
\begin{align*}
+\infty&=\sum_{n\notin\mathcal{N},n\geq 2}\frac{\ln^2 Q_{n+1}\ln^{1+2\varepsilon}{Q_{n}}}{Q^2_n }\\ &=\sum_{n\notin\mathcal{N},n\geq 2}\frac{\ln^2 Q_{n+1}(\ln{Q_{n}})^{2+3\varepsilon}}{Q^2_n \ln^{1+\varepsilon}{Q_{n}}} \\
&\leq \sum_{n\notin\mathcal{N},n\geq 2}\frac{\ln^2 Q_{n+1}(\ln{\ln Q_{n+1}})^{2+4\varepsilon}}{Q^2_n \ln^{1+\varepsilon}{Q_{n}}}\end{align*} which consequently gives us the desired result.
\end{proof} 
We are now ready to prove our first main result.
\begin{proof}[{Proof of Theorem \ref{t:mainB}}]  Consider the set
 $$E:=\left\{ \  \alpha\in [0,1]: \  \sum_{n=1}^\infty\frac{\ln Q_{n+1}}{Q_{n} }=+\infty  \right\} .$$ 
Since $C_\sigma$-capacity is countably sub-addittive it is enough to prove that $C_\sigma(E)=0$ for any $\sigma>2$.  Let $K\subset E$ be any compact set. Since, the Brjuno set is a Borel set, by property 1 of the $C_\sigma$-capacity, it is sufficient to show that $C_\sigma(K)=0$ for any $\sigma>2$.

For any $\alpha\in K$ and for any $\varepsilon>0$, we will prove that $$\sum_{q=10}^\infty\sum_{p=1}^{q-1}\frac{\ln^2{\left|\alpha-\frac{p}{q}\right|}(\ln{\ln{(e+\frac{1}{|\alpha-\frac{p}{q}|})}})^{2+4\varepsilon}}{q^2(\ln{q})^{1+\varepsilon}}=+\infty.$$

Since all terms are nonnegative, by using properties of continued fractions,  we bound  the sum below by focusing on $\frac{P_n}{Q_n}$:
\begin{align*}
\sum_{q=10}^\infty\sum_{p=1}^{q-1}\frac{\ln^2{\left|\alpha-\frac{p}{q}\right|}(\ln{\ln{(e+\frac{1}{|\alpha-\frac{p}{q}|})}})^{2+4\varepsilon}}{q^2({\ln{q}})^{1+\varepsilon}}&\ge\sum_{n=2}^\infty\frac{\ln^2 \left|\alpha-\frac{P_n}{Q_{n}}\right|(\ln{\ln{(e+\frac{1}{|\alpha-\frac{P_n}{Q_{n}}|})}})^{2+4\varepsilon}}{Q^2_n(\ln Q_{n})^{1+\varepsilon}}\\
&+\sum_{n=1}^\infty\sum_{p\neq P_n,p=1}^{Q_n-1}\frac{\ln^2{\left|\alpha-\frac{p}{Q_{n}}\right|}(\ln{\ln{(e+\frac{1}{|\alpha-\frac{p}{Q_{n}}|})}})^{2+4\varepsilon}}{Q_{n}^2(\ln{Q_{n}})^{1+\varepsilon}}\\ &+\sum_{q\neq Q_{n+1},q=10}^\infty\sum_{p=1}^{q-1}\frac{\ln^2{\left|\alpha-\frac{p}{q}\right|}(\ln{\ln{(e+\frac{1}{|\alpha-\frac{p}{q}|})}})^\sigma}{q^2(\ln{q})^{1+\varepsilon}}\\
&\ge\sum_{n=2}^\infty\frac{\ln^2 (Q_{n}Q_{n+1})(\ln{\ln{(Q_{n}Q_{n+1}))}})^{2+4\varepsilon}}{Q^2_n(\ln Q_{n})^{1+\varepsilon}}\\ &\geq\sum_{n=2}^\infty\frac{\ln^2 (Q_{n+1})(\ln{\ln{Q_{n+1}}})^{2+4\varepsilon}}{Q^2_n (\ln Q_{n})^{1+\varepsilon}}=+\infty.
\end{align*}
The last equality is thanks to Lemma \ref{claim:brjuno}.

Now, we define the following Borel measure
\begin{align}\label{Bme}
\mu=\sum_{q=10}^\infty\sum_{p=1}^{q-1}\frac{\delta_{\frac{p}{q}}}{q^2\ln^{1+\varepsilon}{q}},\end{align}
where $\delta_{\frac{p}{q}}$ is a Dirac measure supported at $\frac{p}{q}$. Then, $\text{supp}\mu\subset[0,1]$, and we have
$$|\mu|\le\sum_{q=10}^\infty\frac{1}{q\ln^{1+\varepsilon}{q}}<+\infty.$$

For each $\alpha\in K$, the potential of $\mu$ at $\alpha$ is
\begin{align*}U^\mu(\alpha)&=\sum_{q=10}^\infty\sum_{p=1}^{q-1}\frac{\ln^2{\left|\alpha-\frac{p}{q}\right|}(\ln{\ln{(e+\frac{1}{|\alpha-\frac{p}{q}|})}})^{2+4\varepsilon}}{q^2(\ln{q})^{1+\varepsilon}}\\ &\ge\sum_{n=1}^\infty\frac{\ln^2\ln Q_{n+1}\ (\ln\ln{\ln{Q_{n+1}}})^{2+4\varepsilon}}{Q^2_n(\ln Q_{n})^{1+\varepsilon}}=+\infty.  
\end{align*}

Thus, by property 2 of $C^1_\sigma$-capacity $C^1_{2+4\varepsilon}(K)=0$ for any $\varepsilon>0$ and hence $C^1_{\sigma}(K)=0$ for any  $\sigma>2$. This implies that $C^1_\sigma(E)=0$, for any $\sigma>2$. 
The last assertion follows from property 4 of $C_\sigma^1$-capacity.
\end{proof}
\begin{remark}
In \cite{KR2}, Rakhimov and the first author proved an analogue of the Sadullaev–Rakhimov theorem in higher dimensions. However, in higher dimensions, there is no number-theoretical approach to the Brjuno set; therefore, a priori, the idea we used here does not work in higher dimensions.
\end{remark}
For the proof of Theorem \ref{t:mainPM}, we need the following lemma, whose proof is apparently the same as that of Lemma \ref{claim:brjuno}.
\begin{lemma}\label{lemma:pm}Let $\alpha$ be an irrational number and $Q_n$ be as in \eqref{eq:Qn}. If 
  $$\sum_{n=1}^\infty\frac{\ln{\ln Q_{n+1}}}{Q_n}=+\infty$$
  then, for any $\varepsilon>0$, we have
       $$ \sum_{n=1}^\infty\frac{\ln^2\ln Q_{n+1}(\ln{\ln\ln{Q_{n+1}}})^{2+4\varepsilon}}{Q^2_n \ln^{1+\varepsilon}{Q_{n}}}=+\infty.$$ 
  \end{lemma}
The proof mirrors Lemma \ref{claim:brjuno}, using the set
  $$\mathcal{N'}:=\left\{n\in\mathbb{N} :\ (\ln\ln\ln Q_{n+1})^{2+4\varepsilon}<(\ln Q_{n})^{2+3\varepsilon} \right\}$$   
and inequality
$$\ln \ln Q_{n+1}<e^{(\ln Q_{n})^\beta}$$ for any $n\in\mathcal{N'}$, $\beta=\frac{2+3\varepsilon}{2+4\varepsilon}<1$.

Let us now prove our second main result.
\begin{proof}[{Proof of Theorem \ref{t:mainPM}}]
The proof parallels to the proof of Theorem \ref{t:mainB}. 
For $\alpha\in\mathbb{R}\setminus\mathcal{PM}$ we have
$$\sum_{n=1}^\infty\frac{\ln\ln Q_{n+1}}{Q_n}=+\infty.$$
Again, it is enough to prove that $C_\sigma^2(E)=0$ for 
$$E:=\left\{ \  \alpha\in [0,1]: \  \sum_{n=1}^\infty\frac{\ln \ln Q_{n+1}}{Q_{n} }=+\infty  \right\} .$$ 
The Perez-Marco set $\mathcal{PM}\subset\mathbb{R}$ is also a Borel set. Therefore, by the property 1 of the $C_\sigma$-capacity, it is enough that $C^2_\sigma(K)$ for any compact $K\subset E$ and for any fixed $\sigma>2$. 

For any $\varepsilon>0$, we will show that 
\begin{equation}\label{eq:pminfty}
\sum_{q=10}^\infty\sum_{p=1}^{q-1}\frac{ \ln^2\ln(e+\frac{1}{ |\alpha-\frac{p}{q}|}) \ (\ln{\ln\ln{(e^3+\frac{1}{|\alpha-\frac{p}{q}|})}})^{2+4\varepsilon}}{q^2(\ln{q})^{1+\varepsilon}}=+\infty
\end{equation}
 for any $\alpha\in E$.
  
Fix $\alpha\in E$. Let $\frac{P_n}{Q_n}$ be as in \eqref{eq:Qn}. As before, we have the following inequality
\begin{align*}
\sum_{q=10}^\infty\sum_{p=1}^{q-1}\frac{ \ln^2\ln (e+\frac{1}{ |\alpha-\frac{p}{q}|}) \ (\ln{\ln\ln{( e^3+\frac{1}{|\alpha-\frac{p}{q}|})}})^{2+4\varepsilon}}{q^2({\ln{q}})^{1+\varepsilon}}\ge\sum_{n=2}^\infty\frac{\ln^2\ln Q_{n+1}\ (\ln\ln{\ln{Q_{n+1}}})^{2+4\varepsilon}}{Q^2_n(\ln Q_{n})^{1+\varepsilon}}.
\end{align*}
Consequently, thanks to Lemma \ref{lemma:pm} we have \eqref{eq:pminfty}. Define the following Borel measure
$$
\mu=\sum_{q=10}^\infty\sum_{p=1}^{q-1}\frac{\delta_{\frac{p}{q}}}{q^2\ln^{1+\varepsilon}{q}}.$$
Then by \eqref{eq:pminfty} the potential of $\mu$ at $\alpha$ is    
    $$U^2_\mu(\alpha)=\sum_{q=10}^\infty\sum_{p=1}^{q-1}\frac{ \ln^2\ln (e+\frac{1}{ |\alpha-\frac{p}{q}|}) \ (\ln{\ln\ln{( e^3+\frac{1}{|\alpha-\frac{p}{q}|})}})^{2+4\varepsilon}}{q^2({\ln{q}})^{1+\varepsilon}}=+\infty.$$
Thus, by property 2 of $C^2_\sigma$ capacity $C^2_{2+4\varepsilon}(K)=0$ for any $\varepsilon>0$ and hence $C^2_{\sigma}(K)=0$ for any  $\sigma>2$. This implies that $C^2_\sigma(E)=0$, for any $\sigma>2$. 
The last assertion follows from property 4 of $C_\sigma^2$-capacity.  
\end{proof}

\end{document}